\documentclass[12pt,a4paper]{amsart}

\usepackage[latin1]{inputenc}
\usepackage{amsmath,amssymb,amsthm,bm}

\newcommand{\no}[1]{\left\| #1 \right\|}

\numberwithin{equation}{section}

\newtheorem{theorem}[equation]{Theorem}

\newtheorem{lemma}[equation]{Lemma}
\newtheorem{corollary}[equation]{Corollary}
\newtheorem{proposition}[equation]{Proposition}

\theoremstyle{definition}
\newtheorem{definition}[equation]{Definition}

\newtheorem{remark}[equation]{Remark}

\newcommand{\cal}{\mathcal}
\newcommand{\un}{\underline}
\newcommand{\N}{\mathbb N}
\newcommand{\NN}{\mathbf N}

\newcommand{\C}{\mathbb C}

\newcommand{\lin}{{\rm lin}}

\newcommand{\<}{\langle}
\renewcommand{\>}{\rangle}

\newcommand{\CHI}[1]{\ensuremath{ \chi\raisebox{-1ex}{$\scriptstyle #1$} }}

\begin{document}
\title
{Modules, completely positive maps, and a generalized KSGNS
construction} \maketitle

\begin{center}

Juha-Pekka Pellonp\"a\"a \\
Turku Centre for Quantum Physics \\
Department of Physics and Astronomy\\ University of Turku\\ FI-20014 Turku, Finland\\  \texttt{juhpello@utu.fi}\\ \mbox{}\\

Kari Ylinen\\
Department of Mathematics\\ University of Turku\\ FI-20014 Turku,
Finland\\ \texttt{ylinen@utu.fi}

\end{center}

\begin{abstract}
A very general KSGNS type dilation theorem in the context
of right (not necessarily Hilbert) modules over $C^*$-algebras is presented.
The proof uses Kolmogorov type decompositions for positive-definite
kernels with values in spaces of sesquilinear maps.
More specific functional analytic applications 
are obtained by adding assumptions. 
\end{abstract}

\subjclass{46L08 (Primary); 47A07 (Secondary) 



\section{Introduction}

The celebrated dilation theorem of Stinespring \cite{Stinespring}
subsumed Naimark's dilation theorem for semispectral measures and
the fundamental Gelfand-Naimark-Segal construction of
representations of $C^*$-algebras based on positive linear forms.
Stinespring's work has had no shortage of applications or
extensions. One notable generalization is known under the name of
the KSGNS construction, the ''K'' referring to Kasparov, see e.g.
\cite[Chapter 5]{La95}. A very recent paper in this general
direction is \cite{Asadi09}, and we also mention specifically
\cite{Murphy} which is related to the line of approach taken in this
paper in a more general situation.

There are various even physically motivated reasons to relax the
framework of the dilation theorems. In \cite{HPY1} and \cite{HPY2}
this was done by considering sesquilinear form valued measures as
generalizations of the more traditional theory of operator measures
and their dilations. The physical background is related to the need
to describe measurement situations where only a restricted class of
state preparations are available.

The present paper grew out of an attempt to deal jointly with some
challenges present in the lines of development referred to in both
of the preceding paragraphs. We consider both Hilbert $C^*$-modules
and substantially more general structures. Some highlights of our
results are the following.

Section 3 contains  Kolmogorov type decompositions for
positive-definite kernels with values in spaces of sesquilinear
maps. In the algebraic version, Proposition 3.1, the relevant maps
are $\C$-sesquilinear on $V\times V$ for a vector space $V$ and take
values in a ${}^*$-algebra. More specific information -- the
existence and uniqueness of a minimal Kolmogorov decomposition -- is
obtained in Theorem 3.3 dealing with a (not necessarily Hilbert) 
module $V$ over
a $C^*$-algebra $A$ and $A$-valued $A$-sesquilinear maps on $V\times
V$.

The Kolmogorov decompositions of Section 3 provide a central
technique for the sequel. In particular, Theorem 4.3 and Corollary
4.8 are very general  KSGNS type results based on this approach. It
should be mentioned that the general outline of these arguments is
largely inspired Murphy's work in \cite{Murphy}. The case where the
coefficient $C^*$-algebra $A$ has no identity element causes extra
difficulties which we address by resorting to the second adjoint
$A^{**}$ (see Theorem 4.7).

At the end of the paper we return to direct generalizations of the
motivating case of an operator measure. In Section 5 we prove in the
context of modules and sesquilinear maps an extension of the fact
already shown by Stinespring that for a commutative domain
positivity implies complete positivity (Theorem 5.2), and a rather
concrete example concludes the paper.

\section{Basic definitions for modules and mappings}
In this section we summarize some of the basics of the theory of
($C^*$-)modules. We use \cite{La95}, \cite{G-BVF01}, and
\cite{Manuilov} as general sources where one can find the proofs of
results which  we use without explicit reference. 
The scalar field of all vector spaces is
$\C$. We denote the scalar multiplication of any vector space $V$ by
$cv$ where $c\in\C$, $v\in V$, and we let $I_V$ denote the identity
operator $v\mapsto v$ of a vector space $V$.

For any algebra $A$, the algebra product is denoted by $aa'$ where
$a,\,a'\in A$. If $A$ is a *-algebra, the involution of $A$ is
denoted by $a^*$ where $a\in A$. We write $a\ge 0$ and call $a$ {\it
positive} if for some $p\in\N:=\{1,2,3,\ldots\}$ exist elements
$a_m\in A$, $m=1,\dots,p$, such that $\sum_{m=1}^p a_m^*a_m=a$. If
$A$ is a $C^*$-algebra, we let $\|a\|_A$ denote the norm of $a$.


A set $V$ is a {\it (right) module over an algebra $A$} or, briefly,
an {\it $A$-module}
if it satisfies the  following axioms:
\begin{enumerate}
\item $V$ is a vector space.
\item There exists a mapping (module product)
$V\times A\ni(v,a)\mapsto v\cdot a\in V$ which satisfies the
following requirements for all $v,\,v'\in V$, $a,\,a'\in A$, and
$c\in\C$:
\begin{enumerate}
\item $(v+v')\cdot a=v\cdot a+v'\cdot a,$
\item $v\cdot(a+a')=v\cdot a+v\cdot a',$
\item $v\cdot(aa')=(v\cdot a)\cdot a',$
\item $c(v\cdot a)=(cv)\cdot a=v\cdot(ca)$,
\item if $A$ has an identity $e$ then $v\cdot e=v$.
\end{enumerate}
\end{enumerate}

(We do not consider left (or bi-) modules in this paper.) We say
that a vector subspace $W$ of an $A$-module $V$ is an {\it
$A$-submodule} of $V$ if $W\cdot A\subseteq W$. If $V$ is an
$A$-module and $W$ its $A$-submodule, one can define the {\it
quotient $A$-module} $V/W=V/\sim$ as the quotient vector space
(consisting of equivalence classes $[v]=v+W$ with respect to the
equivalence relation $v\sim v'$ if and only if $v-v'\in W$) equipped
with the (well-defined) module product $[v]\cdot a:=[v\cdot a]$.

Let $A$ be a *-algebra and $V$ an $A$-module. A {\it semi-inner
product} is a mapping $V\times V\ni(v,v')\mapsto\<v|v'\>\in A$ for
which the following conditions hold for all $v,\,v',\,v''\in V$,
$a\in A$, and $c\in\C$:
\begin{enumerate}
\item $\<v|v\>\ge 0,$
\item $\<v|v'+cv''\>=\<v|v'\>+c\<v|v''\>,$
\item $\<v|v'\cdot a\>=\<v|v'\>a,$
\item $\<v|v'\>=\<v'|v\>^*.$
\end{enumerate}
If, in addition, $\<v|v\>=0$ implies  $v=0$ (the definiteness
axiom), we say that $(v,v')\mapsto\<v|v'\>$ is an {\it inner
product}. An $A$-module $V$ equipped with a (semi-)inner product is
called as a {\em (semi-) inner product $A$-module}.

Let now $A$ be a $C^*$-algebra. For any semi-inner product
$A$-module $V$ we have the Cauchy-Schwarz  inequality
$$
\<v|v'\>\<v'|v\>\le\|\<v'|v'\>\|_A\<v|v\>,\hspace{1cm}v,\, v'\in V.
$$
If an inner-product $A$-module $V$ is complete with respect to the
norm $\|v\|:=\sqrt{\|\<v|v\>\|_A}$ we say that $V$ is  a {\em
Hilbert $C^*$-module (over $A$)} or a {\em Hilbert $A$-module} and
usually denote $V$ by $M$. Next we define the basic sets of mappings
which we are going to use later.

Let $A$ be an algebra, and let $V$ and $W$ be vector spaces. We let
${\rm Lin}_\C(V,W)$ denote the set of $\C$-linear mappings from $V$
to $W$. If $V$ and $W$ are $A$-modules, we let ${\rm Lin}_A(V,W)$
denote the set of {\em $A$-linear} mappings $f:V\to W$; by
definition,  $f\in {\rm Lin}_A(V,W)$ if $f(v\cdot a)=f(v)\cdot a$
for all $v\in V$ and $a\in A$, and moreover $f\in {\rm
Lin}_\C(V,W)$. 
If $A$ is unital, the latter requirement may be replaced by
additivity. If $M$ is a Hilbert $C^*$-module over a $C^*$-algebra
$A$, then we let $L_A(M)$ denote the set of adjointable maps from
$M$ into itself. It is known that $L_A(M)\subseteq{\rm Lin}_A(M,M)$
and every element of $L_A(M)$ is bounded.

We let  $\operatorname{Lin}^\times_{\C}(V,A)$ denote the
$\C$-linear space of $\C$-antilinear (i.e., conjugate linear)
mappings from a vector space $V$ to an algebra $A$, that is, $f\in
\operatorname{Lin}^{\times}_{\C}(V,A)$ if and only if
$f(v+v')=f(v)+f(v')$ and $f(cv)=\overline{c}f(v)$ for all $v,\,v'\in
V$ and $c\in\C$. For any $v\in V$ and $f\in
\operatorname{Lin}^{\times}_{\C}(V,A)$ we denote $\<v,f\>:=f(v)$.

Let $\operatorname{Lin}^\times_{A}(V,A)$ be the $\C$-linear space of
the {\it $A$-antilinear} mappings $f$ from an $A$-module $V$ to a
*-algebra $A$; by definition,
$f\in\operatorname{Lin}^\times_{\C}(V,A)$ and we have also $f(v\cdot
a)=a^*f(v)$ for all $a\in A$ and $v\in V$.  If there is no
possibility of confusion, we may denote
$\operatorname{Lin}^{\times}_{\C}(V,A)$ by $V^\times_\C$  and
$\operatorname{Lin}^{\times}_{A}(V,A)$ by $V^\times_A$ or simply by
$V^\times$.

Let $V$ be a module over a *-algebra $A$. Now $V^\times=V^\times_A$
becomes an $A$-module when we define the module product as
$V^\times\times A\ni(f,a)\mapsto fa\in V^\times$ where
$(fa)(v):=f(v)a$. Thus, we have an $A$-sesquilinear pairing
$$
V\times V^\times\ni(v,f)\mapsto \<v,f\>=f(v)\in A;
$$
by definition, it is $\C$-sesquilinear and satisfies $\<v\cdot
a,f\>=a^*\<v,f\>$ and $\<v,fa\>=\<v,f\>a$ for all $v\in V$, $f\in
V^\times$, and $a\in A$. If $V$ has an inner product
$\<\,\cdot\,|\,\cdot\,\>$, one can then $A$-linearly embed  $V$ in
$V^\times$ by $V\ni v\mapsto f_v\in V^\times$ where
$f_v(v'):=\<v'|v\>$ for all $v'\in V$. Note that $f_v=f_{v''}$
implies that $v=v''$ by the definiteness of the inner product, and
the above pairing is consistent with the embedding $v\mapsto f_v$:
$$
\<v',f_v\>=f_v(v')=\<v'|v\>,\hspace{1cm} \<v'\cdot a,f_v
a'\>=a^*\<v',f_v\>a'=a^*\<v'|v\>a'=\<v'\cdot a|v\cdot a'\>.
$$

Finally, for any vector space $V$, we let $S_\C(V;A)$ be the vector
space of $\C$-sesquilinear maps $V\times V\to A$; we always assume
$s\in S_\C(V;A)$ to be antilinear with respect to the first
argument. Usually, we write briefly $S_\C(V)$ instead of
$S_\C(V;A)$. If $V$ is a module over a *-algebra $A$, we let
$S_A(V)$ be the vector space of {\it $A$-sesquilinear} maps
$s:V\times V\to A$, that is, $s\in S_\C(V)$  and $s(v,v'\cdot
a)=s(v,v')a$ and $s(v\cdot a,v')=a^*s(v,v')$ for all $v,\,v'\in V$
and $a\in A$. Note that any $s\in S_A(V)$ can be considered as an
$A$-linear mapping $V\to V^\times_A$, $v'\mapsto\big[v\mapsto
s(v,v')\big]$ and we may identify $S_A(V)$ with ${\rm
Lin}_A(V,V^\times_A)$.

\section{A generalization of the Kolmogorov decomposition}
Let $A$ be a *-algebra, and let $n\in\N$. Let $M_n(A)$ be the matrix
algebra consisting of the $A$-valued $n\times n$-matrices
$(a_{ij})_{i,j=1}^n$. Then $M_n(A)$ is a *-algebra (even a
$C^*$-algebra if $A$ is one), and  its positive elements are by
definition finite sums of matrices of the form $M^*M$ with $M\in
M_n(A)$. Thus  an element $(a_{ij})_{i,j=1}^n$ of $M_n(A)$ is
positive if and only if there are elements $a_{mki}\in A$,
$m\in\{1,\dots,p\}$, $k,\,i\in\{1,\dots,n\}$, such that
$a_{ij}=\sum_{m=1}^p\sum_{k=1}^n a_{mki}^* a_{mkj}$. As in the proof
of Lemma 3.1 in \cite[p.\ 193]{Takesaki}  it follows that
$(a_{ij})_{i,j=1}^n\in M_n(A)$ is positive if and only if it is a
finite sum of matrices of the form $(a_i^*a_j)_{i,j=1}^n$ with
$a_1,\dots,a_n\in A$.
This implies that $\sum_{i,j=1}^n \un a_i^*a_{ij}\un a_j\ge 0$ for
all $\un a_1,\ldots,\un a_n\in A$. (If $A$ is a $C^*$-algebra, the
converse also holds, see Lemma 3.2 in \cite[p.\ 193]{Takesaki}.)
Note that, for any positive matrix $(a_{ij})_{i,j=1}^n$,
$a_{ij}^*=a_{ji}$ and $\sum_{i,j=1}^n a_{ij}\ge 0$.

Let $V$ be a vector space [resp.\ an $A$-module] and
$M_n\big(S_\C(V)\big)$ [resp.\ $M_n\big(S_A(V)\big)$] the
$\C$-linear space of $n\times n$-matrices $(s_{ij})_{i,j=1}^n$ where
the matrix elements $s_{ij}$ belong to $S_\C(V)=S_\C(V;A)$ [resp.\
to $S_A(V)$]. Note that the matrix multiplication is not defined.

Unless stated otherwise,  throughout this paper  $V$ is only assumed
to be a vector space, and {\it when we need a module structure, then
$V$ is an $A$-module.} For example, $M_n\big(S_A(V)\big)\subseteq
M_n\big(S_\C(V)\big)$ holds only when $V$ is an $A$-module so that
the left hand side makes sense.

We say that  $(s_{ij})_{i,j=1}^n\in M_n\big(S_\C(V)\big)$ is {\it
positive} if the matrix $\big(s_{ij}(v_i,v_j)\big)_{i,j=1}^n$ is a
positive element of $M_n(A)$ whenever $v_1,\dots,v_n\in V$.


For any set $X\ne\emptyset$, we say that a mapping $K:\,X\times X\to
S_\C(V)$ is a {\it positive-definite kernel} if, for all $n\in\N$
and $x_1,x_2,\ldots,x_n\in X$, the matrix $\big(K(x_i,x_j)
\big)_{i,j=1}^n\in M_n\big(S_\C(V)\big)$ is positive. If
$K:\,X\times X\to S_A(V)$ is a positive-definite kernel then we say
that $K$ is a {\it (positive-definite) $A$-kernel}. Next we
construct a Kolmogorov type decomposition for a positive-definite
kernel. We also consider the special case where the kernel is an
$A$-kernel.

Let $K:\,X\times X\to S_\C(V)$ be a positive-definite kernel. Let
$V^X_{\rm fin}$ be the  $\C$-linear space of functions $f:\,X\to V$
such that $f(x)\ne 0$ only for finitely many $x\in X$, and let
$V_\C^{\times X}$ [resp.\ $V^{\times X}_A$] be the $\C$-linear space
of functions $f:\,X\to V^\times_\C$ [resp.\ $f:\,X\to V^\times_A$].
Define a $\C$-linear mapping $\tilde K:\,V^X_{\rm fin}\to V^{\times
X}_\C$ by
$$
\<v,\tilde Kf(x)\>:=\sum_{x'\in
X}K(x,x')\big(v,f(x')\big),\hspace{1cm}v\in V,\;f\in V^X_{\rm
fin},\;x\in X.
$$
Let ${\rm Im}\,\tilde K$ be the image of $\tilde K$. Consider the
mapping $\<\,\cdot\, | \, \cdot\,\>$ defined as
$$
{\rm Im}\,\tilde K\times {\rm Im}\,\tilde K\ni (\tilde Kf,\tilde
Kf')\mapsto\<\tilde Kf|\tilde Kf'\>:= \sum_{x,x'\in
X}K(x,x')\big(f(x),f'(x')\big)\in A.
$$
It is easy to see that it is positive (i.e.\ $\<\tilde Kf|\tilde
Kf\>\ge 0$), $\C$-sesquilinear, conjugate symmetric, and well
defined since, if $\tilde Kf=\tilde Kf'$, $f,\,f'\in V^X_{\rm fin}$,
then $ \sum_{x'\in X}K(x,x')\big(v,f(x')-f'(x')\big)=0 $ for all
$v\in V$, $x\in X$, implying that
$$
\<\tilde Kf''|\tilde Kf-\tilde Kf'\>=\<\tilde Kf-\tilde Kf'|\tilde
Kf''\>^*= \sum_{x,x'\in X}K(x,x')\big(f''(x),f(x')-f'(x')\big)=0
$$
for all $f''\in V^X_{\rm fin}$.

Let $x\in X$ and define a $\C$-linear mapping $D(x):\,V\to {\rm
Im}\,\tilde K$ by
$$
D(x)v:=\tilde K f^v_x,\hspace{1cm}v\in V
$$
where $f^v_x(x):=v\in V$ and $f^v_x(x'):=0$ for all $x'\ne x$. Then
$$
\<D(x)v|D(x')v'\>=\<\tilde K f^v_x|\tilde K
f^{v'}_{x'}\>=K(x,x')(v,v')
$$
for all $x,\,x'\in X$ and $v,\,v'\in V$. Thus (by denoting $W={\rm
Im}\,\tilde K$) we have:

\begin{proposition}\label{p21}
Let $A$ be a *-algebra and $V$ a vector space. Let $X\neq\emptyset$
be a set and $K:\,X\times X\to S_\C(V;A)$ a positive-definite
kernel. There exists a vector space $W$ equipped with a positive,
$A$-valued, $\C$-sesquilinear form $\<\,\cdot\,|\,\cdot\,\>$, and a
mapping $D:\,X\to{\rm Lin}_\C(V,W)$ such that
$$
K(x,x')(v,v')=\<D(x)v|D(x')v'\>,\hspace{5mm}x,\,x'\in X,\;v,\,v'\in
V.
$$
We say that $(W,D)$ is a {\em Kolmogorov decomposition} for $K$.
\end{proposition}

\begin{remark}\label{rema}
{\it If $K$ is an $A$-kernel,} then ${\rm Im}\,\tilde K$ is
contained in $V^{\times X}_A$ and, since $V$ and $V^\times_A$ are
$A$-modules, by defining the module products pointwise, $V^X_{\rm
fin}$, $V^{\times X}_A$ and ${\rm Im}\,\tilde K$ become $A$-modules;
then also $\tilde K$ is $A$-linear, that is,
$$
\<v,\tilde K(f\cdot a)(x)\>=\sum_{x'\in X}K(x,x')\big(v,f(x')\cdot
a\big)= \sum_{x'\in X}K(x,x')\big(v,f(x')\big)a=\<v,[(\tilde
Kf)a](x)\>,
$$
and $\<\,\cdot\, | \, \cdot\,\>$ is $A$-sesquilinear and thus a
semi-inner product. Moreover, $D(x)$ can be considered as an
$A$-linear mapping from $V$ to ${\rm Im}\,\tilde K$.
\end{remark}

\subsection{The $C^*$-algebra case}
{\it In this subsection we assume that $A$ is a $C^*$-algebra and
$K$ is an $A$-kernel.} Now $\<\,\cdot\, | \, \cdot\,\>$ is an inner
product. Indeed, by remark \ref{rema}, we need only to prove the
definiteness condition: Assume that $\<\tilde Kf|\tilde Kf\>=0$
where $f\in V^X_{\rm fin}$. From the Cauchy-Schwarz inequality, it
follows that $\<\tilde Kf'|\tilde Kf\>=0$ for all $f'\in V^X_{\rm
fin}$. Especially, this holds for $f'=f^v_x$. Since
$$
\<\tilde Kf'|\tilde Kf\>= \sum_{x'\in X}\left[\sum_{x\in X}
K(x',x)\big(f'(x'),f(x)\big)\right] =\sum_{x'\in X}\<f'(x'),\tilde
Kf(x')\>,
$$
by putting $f'=f^v_x$, we get
$$
\<\tilde Kf^v_x|\tilde Kf\>=\<v,\tilde Kf(x)\>.
$$
Now, since $\<v,\tilde Kf(x)\>=0$ for all $v\in V$, we get $\tilde
Kf(x)=0$. But this holds for all $x\in X$ implying that $\tilde
Kf=0$. The above inner product defines a norm
$Kf\mapsto\sqrt{\|\<\tilde Kf|\tilde Kf\>\|_A}$. Let $M$ be the
completion of ${\rm Im}\,\tilde K$ with respect to the above norm.
It is straightforward to verify $M$ is a Hilbert $C^*$-module over
$A$ \cite[p.\ 4]{La95}.

Let $x\in X$ and define a 'dual mapping' $D(x)^\times:\,M\to
V_A^\times$ of $D(x):\,V\to M$ as follows:
$$
\<v,D(x)^\times m\>:=\<D(x)v|m\>=\<\tilde K
f^v_x|m\>,\hspace{1cm}v\in V,\;m\in M.
$$
It is clearly $A$- (and hence $\C$-)linear. For example, if
$m=D(x')v'=\tilde K f^{v'}_{x'}$ we get
$$
\<v,D(x)^\times D(x')v'\> =\<\tilde K f^v_x|\tilde K f^{v'}_{x'}\>
=K(x,x')(v,v'),\hspace{1cm}x,\,x'\in X,\;v,\,v'\in V,
$$
or, briefly,
$$
K(x,x')=D(x)^\times D(x')
$$
for all $x,\,x'\in X$. Next we study the minimality of $M$.

For any $f\in V_{\rm fin}^X$ we can write $f=\sum_{i=1}^k
f^{v_i}_{x_i}$ where $v_i\in V$ and $x_i\in X$ for all
$i=1,2,\ldots,k$, $k\in\N$. Since $\tilde Kf=\sum_{i=1}^k\tilde
Kf^{v_i}_{x_i}=\sum_{i=1}^k D(x_i)v_i$,
$$
\tilde Kf\in \lin_\C\cup_{x\in X}D(x)V=\lin_\C\{D(x)v\,|\,x\in
X,\,v\in V\}=\lin_\C\{\tilde Kf^v_x\,|\,x\in X,\,v\in V\}
$$
and, thus, $\lin_\C\cup_{x\in X}D(x)V$ is dense in $M$ (minimality).
Thus, we have proved the first part of the following generalization
of the Kolmogorov decomposition:
\begin{theorem}\label{t1}
Let $V$ be a module over a $C^*$-algebra $A$. Let $X\neq\emptyset$
be a set and $K:\,X\times X\to S_A(V)$ a positive-definite
$A$-kernel. There exists a Hilbert $C^*$-module $M$ over $A$ and a
mapping $D:\,X\to{\rm Lin}_A(V,M)$ such that
\begin{itemize}
\item[(i)]
$ K(x,x')(v,v')=\<v,D(x)^\times D(x')v'\>,\hspace{5mm}x,\,x'\in
X,\;v,\,v'\in V, $
\item[(ii)]
$\lin_\C\cup_{x\in X}D(x)V$ is dense in $M$.
\end{itemize}
We say that $(M,D)$ is a {\em minimal Kolmogorov decomposition} for
$K$.

If $(M',D')$ is another minimal Kolmogorov decomposition for $K$
then there exists a unitary $U:\,M\to M'$ such that $UD(x)=D'(x)$
for all $x\in X$.
\end{theorem}

\begin{proof}
Let  $(M,D)$ and $(M',D')$ be minimal Kolmogorov decompositions for
$K$. Let $x_1,\ldots,x_n\in X$ and $v_1,\ldots,v_n\in V$. Then
$$
\left\|\sum_{i=1}^n D(x_i)v_i\right\|^2 =\left\|\sum_{i,j=1}^n
\<D(x_i)v_i|D(x_j)v_j\>\right\|_A =\left\|\sum_{i,j=1}^n
K(x_i,x_j)(v_i,v_j)\right\|_A =\left\|\sum_{i=1}^n
D'(x_i)v_i\right\|^2,
$$
and, hence, there is a well-defined isometry from a dense linear
subspace of $M$ to $M'$ that maps any $D(x)v$ to $D'(x)v$. Its
extension $U:\,M\to M'$ is the required unitary map.
\end{proof}

\section{A generalization of the KSGNS construction}

Let $A$ and $B$ be *-algebras and $V$ a vector space. For any
$\C$-linear mapping $E:\,B\to S_\C(V)$ ($=S_\C(V;A)$) and $n\in\N$,
we define its $n^{th}$ {\it amplification} $E^{(n)}:\,M_n(B)\to
M_n\big(S_\C(V)\big)$ as
$E^{(n)}\big((b_{ij})_{i,j}\big):=\big(E(b_{ij})\big)_{i,j}$. For
example, $E^{(1)}=E$. Moreover, we say that $E^{(n)}$ is positive if
$E^{(n)}\big((b_{ij})_{i,j}\big)$ is positive for any positive
$(b_{ij})_{i,j}$. Especially, $E:\,B\to S_\C(V)$ {\em positive}, if
$E(b)\geq0$ for any $b\geq0$ in $B$. Now we are ready to define the
concept of complete positivity:

\begin{definition}
A mapping $E:\,B\to S_\C(V)$ is {\it completely positive} if $E$ is
$\C$-linear and $E^{(n)}$ is positive for any positive integer $n$.
\end{definition}
Let $E:\,B\to S_\C(V)$ be a $\C$-linear mapping. It is completely
positive if and only if  the mapping $\tilde E:\,B\times B\to
S_\C(V)$, $(b,b')\mapsto E(b^*b')$ is a positive-definite kernel.

\begin{lemma}\label{l1}
Let $A$ and $B$ be $C^*$-algebras, and $V$ be an $A$-module. Let
$E:\,B\to S_A(V)$ be a completely positive mapping. There exists a
Hilbert $C^*$-module $M$ over $A$ and a $\C$-linear mapping
$D:\,B\to{\rm Lin}_A(V,M)$ such that
\begin{itemize}
\item[(i)]
$ E(b^*b')(v,v')=\<v,D(b)^\times D(b')v'\>,\hspace{5mm}b,\,b'\in
B,\;v,\,v'\in V, $
\item[(ii)]
$\lin_\C\cup_{b\in B}D(b)V$ is dense in $M$.
\end{itemize}
Moreover, there exists a *-homomorphism $\pi:\,B\to L_A(M)$ such
that
$$
\pi(b)D(b')=D(bb'),\hspace{5mm}b,\,b'\in B.
$$
\end{lemma}

\begin{proof}
Since, by Theorem \ref{t1},
$$
E(b^*b')(v,v')=\<v,D(b)^\times
D(b')v'\>=\<D(b)v|D(b')v'\>,\hspace{1cm}b,\,b'\in B,\;v,\,v'\in V,
$$
it follows that, for all $b,\,\un b,\,b'\in B$, $c\in\C$, and
$v,\,v'\in V$,
\begin{eqnarray*}
\<D(b+c\un b)v|D(b')v'\>&=& E((b+c\un b)^*b')(v,v')
=E(b^*b')(v,v')+\overline c E(\un b^*b')(v,v') \\
&=&\<D(b)v|D(b')v'\>+\overline c \<D(\un
b)v|D(b')v'\>=\<[D(b)+cD(\un b)]v|D(b')v'\>.
\end{eqnarray*}
Since $\lin_\C\cup_{b\in B}D(b)V$ is dense in $M$, the above
calculation implies that
$$
D(b+c\un b)v=[D(b)+cD(\un b)]v
$$
for all $v\in V$, that is, $D:\,B\to{\rm Lin}_A(V,M)$ is
$\C$-linear.

If $B$ is unital, we denote $\tilde B=B$. Otherwise, let $\tilde
B:=B\times\C\cong B+\C e$ be the unitisation of $B$ where the
identity of $\tilde B$ is ${e}:=(0,1)$. Let $u$ be a unitary element
of $\tilde B$ (i.e.\ $uu^*=u^*u={e}$). Since $B$ can be viewed as a
left and right ideal of $\tilde B$ it follows that $ub\in B$ for all
$b\in B$. Define a ($\C$-linear) mapping
$$
D':\,B\to{\rm Lin}_A(V,M),\;b\mapsto D'(b):=D(ub).
$$
Since
$$
\<D(ub)v|D(ub')v'\>=E(b^*u^*ub')(v,v')=E(b^*b')(v,v'),\hspace{1cm}v,\,v'\in
V,\;b,\,b'\in B,
$$
and $\lin_\C\cup_{b\in B}D(ub)V$ is obviously dense in $M$, the pair
$(M,D')$ is a minimal Kolmogorov decomposition for $\tilde E$. By
the second part of Theorem \ref{t1}, there exists a unitary mapping
$\pi(u):\,M\to M$, such that
$$
\pi(u)D(b)=D'(b)=D(ub)
$$
for all $b\in B$. It is well known that any $b\in B$ can be
represented as a $\C$-linear combination of four unitaries of
$\tilde B$: $b=\sum_{i=1}^4 c_i u_i$. Define
$$
\pi(b):=\sum_{i=1}^4 c_i\pi(u_i).
$$
Now, for all $b'\in B$,
$$
\pi(b)D(b')=\sum_{i=1}^4 c_i[\pi(u_i)D(b')]=\sum_{i=1}^4
c_iD(u_ib')= D\left(\sum_{i=1}^4 c_i u_i b'\right)=D(bb')
$$
and, from the density of $\lin_\C\cup_{b\in B}D(b)V$ in $M$, it
follows that
$$
\pi:\,B\to L_A(M),\;b\mapsto\pi(b)
$$
is well defined (i.e., independent of the representation of $b$).
Moreover, by the density argument and the equation
$\pi(b)D(b')=D(bb')$, we see immediately that
$$
\pi(b+b')=\pi(b)+\pi(b'),\; \pi(cb)=c\pi(b),\;
\pi(bb')=\pi(b)\pi(b'),\text{ and }\pi(b)^*=\pi\big(b^*\big)
$$
for all $b,\,b'\in B$ and $c\in\C$; that is, $\pi$ is a
*-homomorphism (and, thus, continuous).
\end{proof}

The following theorem is a generalization of the KSGNS-construction
for completely positive mappings $E:\,B\to L_A(M)$ where $B$ is
unital:

\begin{theorem}\label{p1}
Let $V$ be a module over a $C^*$-algebra $A$, $B$ a {\em unital}
$C^*$-algebra, and $E:\,B\to S_A(V)$ a completely positive mapping.
There exist a Hilbert $C^*$-module $M$ over $A$, a {\em unital}
*-homomorphism $\pi:\,B\to L_A(M)$, and an element $J\in{\rm
Lin}_A(V,M)$ such that
\begin{itemize}
\item[(i)]
$ E(b)(v,v')=\<v,J^\times\pi(b)Jv'\>=\<Jv|\pi(b)Jv'\>,
\hspace{5mm}b\in B,\;v,\,v'\in V, $
\item[(ii)]
$\lin_\C\,\pi(B)JV=\lin_\C \{ \pi (b)Jv\, | \, b\in B,\,v\in V\}$ is
dense in $M$.
\end{itemize}
We say that $(M,\pi,J)$ is a {\em minimal dilation} for $E$.

If $(M',\pi',J')$ is another minimal dilation for $E$ then there is
a unitary mapping $U:\,M\to M'$ such that
$$
\pi'(b)=U\pi(b)U^*, \hspace{5mm}b\in B,
$$
and $J'=UJ$.
\end{theorem}

\begin{proof}
Let $M$, $D$, and $\pi$ be as in Lemma \ref{l1}, and let $e$ and
$I_M$ be the units of $B$ and $L_A(M)$, respectively. Define
$J:=D(e)$. Now, for all $b\in B$ and $v,\,v'\in V$,
\begin{eqnarray*}
&D(b)=D(be)=\pi(b)D(e)=\pi(b)J, \\
&E(b)(v,v')=E(e^*b)(v,v')=\<D(e)v|D(b)v'\>=\<Jv|\pi(b)Jv'\>.
\end{eqnarray*}
In addition, $\lin_\C \{ \pi (b)Jv\, | \, b\in B,\,v\in
V\}=\lin_\C\{D(b)v\, | \, b\in B,\,v\in V\}$ is dense in $M$. Since
$\pi(b)m=\pi(eb)m=\pi(e)\pi(b)m$ for all $b\in B$ and $m\in M$, it
follows, e.g., from the density argument above that $\pi(e)=I_M$.
Next we prove the uniqueness assertion.

Let $(M',\pi',J')$ be a minimal dilation for $E$. For all
$b_1,\ldots,b_n\in B$ and $v_1,\ldots,v_n\in V$, by (i),
\begin{eqnarray*}
\left\< \sum_{i=1}^n\pi'(b_i)J' v_i\bigg| \sum_{j=1}^n\pi'(b_j)J'
v_j \right\> &=& \sum_{i,j=1}^n \left\< J' v_i\big|\pi'(b_i^*b_j)J'
v_j \right\> =
\sum_{i,j=1}^n E(b_i^*b_j)(v_i,v_j) \\
&=& \left\< \sum_{i=1}^n\pi(b_i)J v_i\bigg| \sum_{j=1}^n\pi(b_j)J
v_j \right\>.
\end{eqnarray*}
Hence, the mapping which maps $\sum_{i=1}^n\pi(b_i)J v_i$ to
$\sum_{i=1}^n\pi'(b_i)J' v_i$, is well defined, $A$-linear,
($\C$-linear), isometric, and continuous, and extends to a
continuous map $U:\,M\to M'$ by (ii). Since $U$ is an isometric and
surjective $A$-linear map, it is unitary \cite[Theorem 3.5, p.\
26]{La95}.

For all $b,\,b_1,\ldots,b_n\in B$ and $v_1,\ldots,v_n\in V$,
$$
U\pi(b)U^*\sum_{i=1}^n\pi'(b_i)J' v_i= U\sum_{i=1}^n\pi(bb_i)J v_i
=\sum_{i=1}^n\pi'(bb_i)J' v_i =\pi'(b)\sum_{i=1}^n\pi'(b_i)J' v_i,
$$
and by (ii) it follows that
$$
\pi'(b)=U\pi(b)U^*, \hspace{5mm}b\in B.
$$
Moreover, since $\pi'(e)=I_{M'}$,
$$
UJv=UD(e)v=U\pi(e)Jv=\pi'(e)J'v=J'v,\hspace{5mm}v\in V,
$$
and, thus, $J'=UJ$.
\end{proof}

It is worth noting that, when $B$ is not unital, the preceding
theorem cannot be easily extended to that case. Namely, $V$ is not
assumed to be an inner product module and $E$ is not necessarily a
bounded mapping so that one could extend $E$ to the unitisation
$\tilde B$ by a standard method, namely, by declaring that
$E(e)=\|E\|I_V$ (see, e.g.\ \cite[p.\ 55]{La95}). In the following
subsection we address the question of to what extent our theory can
be extended to the nonunital case and what kind of modifications are
needed.

\subsection{The nonunital case}
Let $A$ and $B$ be $C^*$-algebras. We consider a positive linear map
$E:\,B\to S_\C(V;A)$ where $V$ is a vector space. By polarization,
every linear map $E_{v,{v'}}:B\to A$ defined by $E_{v,{v'}}(b)=
E(b)(v,{v'})$ is continuous. Let $E_{v,{v'}}^{**}:B^{**}\to A^{**}$
be its second adjoint and define $\hat E(b)(v,{v'})$ for any $b\in
B^{**}$ and $v,\,{v'}\in V$ by $\hat
E(b)(v,{v'}):=E_{v,{v'}}^{**}(b)$. When the biduals $A^{**}$ and
$B^{**}$ are regarded as von Neumann algebras in the usual way, we
have thus obtained a positive linear map $\hat E:B^{**}\to
S_\C(V;{A^{**}})$. The next lemma is used to show that $\hat E$ is
completely positive if $E$ is completely positive.

\begin{lemma} If $X_1,\dots, X_p$ are Banach spaces, then all
Banach space norms on $X=X_1\oplus\cdots\oplus X_p$, such that the
canonical embeddings $\eta_i:X_i\to X$ defined by $\eta_i(x_i)=
(0,\dots,0,x_i,0,\dots,0)$ are continuous, are equivalent.
\end{lemma}
\begin{proof} Let $\no{\cdot}$ be any complete norm on $X$ satisfying
the stated requirement. Define
${\no{x}}_\infty:=\max\{\no{x_1},\dots,\no{x_p}\}$ for
$x=(x_1,\dots,x_p)\in X$. Each projection $\operatorname{pr}_i:X\to
X_i$ is continuous from $(X,\no{\cdot}_\infty)$ to $X_i$, and since
each $\eta_i$ is by assumption with respect to $\no{\cdot}$, the
identity map $\sum_{i=1}^p\eta_i\circ\operatorname{pr}_i$ is
continuous from $(X,\no{\cdot}_\infty)$ to $(X,\no{\cdot})$, and so
it is a homeomorphism by the open mapping theorem.
\end{proof}

It follows that the $C^*$-algebra norm on $M_n(A)$ is equivalent to
the norm $\no{\cdot}$ defined as the maximum of the norms of the
matrix entries. The dual of $(M_n(B),\no{\cdot}_\infty)$ is known to
be $(M_n(B^*),\no{\cdot}_1)$ (the $l^1$-direct sum), and the dual of
$(M_n(B^*),\no{\cdot}_1)$ is $(M_n(B^{**}),\no{\cdot}_\infty)$. It
follows from the preceding lemma that the bidual of $M_n(B)$ is
canonically isomorphic to the $C^*$-algebra $M_n(B^{**})$.

Let $E:\,B\to S_\C(V;A)$ be completely positive and
$E^{(n)}:\,M_n(B)\to M_n\big(S_\C(V;A)\big)$ be the $n^{\rm th}$
amplification of $E$ (which is thus supposed to be positive). Since
$M_n(B)^{**}\cong M_n(B^{**})$, the $n^{\rm th}$ amplification $\hat
E^{(n)}$ of $\hat E$ can be viewed as a mapping from $M_n(B)^{**}$
to $M_n\big(S_\C(V;{A^{**}})\big)$. Fix $v_1,\ldots,v_n\in V$. Since
$M_n(B)\ni (b_{ij})_{i,j=1}^n\mapsto \sum_{i,j=1}^n
E(b_{ij})(v_i,v_j)\in A$ is positive and continuous its second
adjoint $M_n(B)^{**}\to A^{**}:\, (b_{ij})_{i,j=1}^n\mapsto
\sum_{i,j=1}^n E_{v_i,v_j}^{**}(b_{ij})=\sum_{i,j=1}^n\hat
E(b_{ij})(v_i,v_j)$ is also. Hence, $\hat E^{(n)}$ is positive and
$\hat E$ completely positive.

Suppose then that $V$ is an $A$-module and $E:\,B\to S_A(V)$ is
completely positive but $B$ {\it is not necessarily unital.} Now
$\hat E:B^{**}\to S_\C(V;{A^{**}})$ is completely positive and
$B^{**}$ is unital. Let $b\in B^{**}$, $v,\,v'\in V$, and $a\in A$.
Since $\hat E(b)(v,{v'})=E_{v,{v'}}^{**}(b)$ it follows that $\hat
E(b)(v,v'\cdot a)=\hat E(b)(v,v')a$ and $\hat E(b)(v\cdot
a,v')=a^*\hat E(b)(v,v')$ and, thus, $\hat E(b)$ is
$A$-sesquilinear.

The algebraic tensor product $V\otimes A^{**}$ becomes an
$A^{**}$-module when one defines a module product as $(v\otimes
a)\cdot a':=v\otimes (aa')$ for all $v\in V$ and $a,\,a'\in A^{**}$.
Define
$$
W:=\lin_\C\big\{(v\cdot a)\otimes a'-v\otimes (aa')\in V\otimes
A^{**}\,\big|\,v\in V,\,a\in A,\,a'\in A^{**}\big\}.
$$
Obviously, $W$ is an $A^{**}$-submodule of $V\otimes A^{**}$ so that
we can define a quotient $A^{**}$-module
$$
\tilde V:=(V\otimes A^{**})/W.
$$
As an $A^{**}$-module, $\tilde V$ is also an $A$-module and one has
the following $A$-linear mapping:
$$
T:\, V\to\tilde V,\; v\mapsto T(v):=[v\otimes e]=v\otimes e+W
$$
where $e$ is the unit of $A^{**}$.

\begin{lemma}\label{3.6}
If $\operatorname{Lin}_{A}(V,A^{**})$ is a separating set for $V$
then $T$ is injective and $V$ can be considered as an $A$-submodule
of $\tilde V$.
\end{lemma}

\begin{proof}
For any $f\in \operatorname{Lin}_{A}(V,A^{**})$ one can define a
$\C$-linear mapping $ \tilde f:\,V\otimes A^{**}\to A^{**} $ by
$\tilde f(v\otimes a'):=f(v)a'$ where $v\in V$ and $a'\in A^{**}$.
Now, for all $v\in V$, $a\in A$ and $a'\in A^{**}$,
$$
\tilde f\big((v\cdot a)\otimes a'\big)=f(v)aa'=\tilde
f\big(v\otimes(aa')\big)
$$
so that $\tilde f(W)=\{0\}$. Suppose then that $v\otimes e\in W$ for
some $v\in V$. Hence $f(v)=\tilde f(v\otimes e)=0=f(0)$ for all
$f\in \operatorname{Lin}_{A}(V,A^{**})$. If
$\operatorname{Lin}_{A}(V,A^{**})$ separates the points of $V$ it
follows that $v=0$ and $T$ is an injection.
\end{proof}

Let then $b\in B^{**}$. Since $\hat E(b)$ is $A$-sesquilinear one
can define an $A^{**}$-sesquilinear mapping $\tilde E(b):\,\tilde
V\times \tilde V\to A^{**}$ by defining that, for all $v,\,v'\in V$
and $a,\,a'\in A^{**}$,
$$
\tilde E(b)(v\otimes a+W,\,v'\otimes a'+W):=a^*\hat E(b)(v,v')a'.
$$
Indeed, $\tilde E(b)$ is well defined since, for all $w=\sum_{i=1}^m
\big[(v_i\cdot a_i)\otimes a'_i-v_i\otimes (a_ia'_i)\big]\in W$ and
$u=\sum_{j=1}^n v_j''\otimes a_j'' \in V\otimes A^{**}$ (here
$v_i,\,v_j''\in V$, $a_i\in A$, and $a_i',\,a_j''\in A^{**}$), one
gets
\begin{eqnarray*}
\tilde E(b)\left(w+W,\,u+W\right) &=& \sum_{i=1}^m \sum_{j=1}^n
\left[(a_i')^*\hat E(b)(v_i\cdot a_i,v''_j)a_j'' -(a_ia_i')^*\hat
E(b)(v_i,v''_j)a_j''\right]=0
\end{eqnarray*}
and $\tilde E(b)\left(u+W,\,w+W\right)=0$. Note that $\tilde
E(b)\big(T(v),T(v')\big)=\hat E(b)(v,v')=E_{v,{v'}}^{**}(b)$ for all
$v,\,v'\in V$. Especially, if $b\in B$ then $\tilde
E(b)\big(T(v),T(v')\big)=\hat E(b)(v,v')=E(b)(v,v')$ for all
$v,\,v'\in V$.

Now the mapping $\tilde E:\,B^{**}\to S_{A^{**}}(\tilde
V),\;b\mapsto \tilde E(b),$ is completely positive. Indeed, let
$(b_{ij})_{i,j}$ be positive $B^{**}$-valued $n\times n$-matrix and
$v_i\in V$ for all $i=1,\ldots,n$. Since the matrix $(\hat
E(b_{ij})(v_i,v_j))_{i,j}$ is positive it follows that $\hat
E(b_{ij})(v_i,v_j)=\sum_{m=1}^n a_{mi}^* a_{mj}$,
$i,\,j\in\{1,\ldots,n\}$, where $a_{mj}\in A^{**}$ for all
$m,\,j\in\{1,\ldots,n\}$. But then the matrix
$$
(i,j)\mapsto \tilde E(b)(v_i\otimes a_i+W,\,v_j\otimes
a_j+W)=\sum_{m=1}^n (a_{mi}a_i)^* a_{mj}a_j
$$
is positive for all $a_i\in A^{**}$, $i=1,\ldots,n$.

We have proved the following theorem:
\begin{theorem}\label{thr}
Let $A$ and $B$ be $C^*$-algebras, $V$ an $A$-module, and $E:\,B\to
S_A(V)$ completely positive. There exist
\begin{enumerate}
\item a module $\tilde V$ over $A^{**}$ and an $A$-linear mapping $T:\,V\to\tilde V$,
\item a completely positive mapping $\tilde E:\,B^{**}\to S_{A^{**}}(\tilde V)$ such that
$E(b)(v,v')=\tilde E(b)\big(T(v),T(v')\big)$ for all $b\in B$ and
$v,\,v'\in V$.
\end{enumerate}
\end{theorem}
Next we study the injectivity of $T$ in view of Lemma \ref{3.6}.

Define an $A$-submodule $V_E^0$ of $V$ as follows:
$$
V_E^0:=\big\{v'\in V\,\big|\,E(b)(v,v')=0,\;b\in B,\,v\in V\big\}.
$$
Let then $V_E:=V/V_E^0$. Since $E(b)(v+u,v'+u')=E(b)(v,v')$ for all
$b\in B$, $v,\,v'\in V$ and $u,\,u'\in V_E^0$ so that, without
restricting generality, we {\it identify} $E$ with a completely
positive mapping
$$
E:\,B\to S_A(V_E).
$$
Moreover, for all nonzero $v'\in V_E$ there exist $\un b\in B$ and
$\un v\in V_E$ such that $E(\un b)(\un v,v')\ne 0$. In other words,
for all nonzero $v'\in V_E$, we have an $f\in
\operatorname{Lin}_{A}(V_E,A)$ such that $f(v')\ne 0$ (e.g.\ choose
$f(v)=E(\un b)(\un v,v)$). Hence, given two $v\ne v'\in V_E$ we have
an $f'\in \operatorname{Lin}_{A}(V_E,A)$ such that $f'(v)\ne f'(v')$
and $\operatorname{Lin}_{A}(V_E,A)$ is a separating set for $V_E$.
By Lemma \ref{3.6} and the identification $V=V_E$, the corresponding
mapping $T$ is an injection and $V_E$ can be considered as an
$A$-submodule of $\tilde V=\tilde V_E$. Theorem \ref{thr} shows
that, without restricting generality, we may always consider
completely positive mappings $E:\,B\to S_A(V_E)$ where $A$ and $B$
are unital $C^*$-algebras or even von Neumann algebras.     

Theorems \ref{p1} and \ref{thr} imply the following corollary:

\begin{corollary}Let $A$ and $B$ be $C^*$-algebras, $V$ an $A$-module,
and $E:\,B\to S_A(V)$ completely positive. There exist a
$C^*$-module $M$ over $A^{**}$, a unital *-homomorphism
$\pi:\,B^{**}\to L_{A^{**}}(M)$, and an element $J\in{\rm
Lin}_{A}(V_E,M)$ such that
$$E(b)(v,v')=\<v,J^\times\pi(b)Jv'\>=\<Jv|\pi(b)Jv'\>$$  for all $b\in B$ and $v,\,v'\in V_E.$
\end{corollary}
Note that $\lin_\C\,\pi(B)JV_E=\lin_\C \{ \pi (b)Jv\, | \, b\in
B,\,v\in V_E\}$ is not necessarily dense in $M$.

\section{Commutativity and complete positivity}

In this section we let $V$ be a vector space and $A$ a
$C^*$-algebra. Let $(\Omega,\Sigma)$ be a measurable space (where
$\Sigma$ is a $\sigma$-algebra of subsets of $\Omega$) and $\cal F$
the commutative $C^*$-algebra of all bounded $\Sigma$-measurable
complex functions on $\Omega$. Let $\CHI X$ be the characteristic
function of a set $X\in\Sigma$. The techniques in the next two
proofs resemble some arguments used in \cite{Y09} in a different
context.

\begin{lemma}\label{lemma:facp} Let $E:\cal F\to S_\C(V;A)$ be a linear map. The
following conditions are equivalent:

{\rm(i)} $E$ is positive;

{\rm(ii)} $E(\CHI X)\geq0$ for all $X\in\Sigma$, and for every $v\in
V$ the map $f\mapsto E(f)(v,v)$ is continuous;

{\rm(iii)} $E$ is completely positive.
\end{lemma}
\begin{proof} Clearly (i) $\implies$ (ii),
since every positive linear map from a $C^*$-algebra into another is
continuous. Now assume (ii). For $i=1,\dots,n$, choose $v_i\in\C$
and let first $f_i$ be a linear combination of characteristic
functions of sets in $\Sigma$. There are disjoint sets $X_k\in
\Sigma$, $k=1,\dots,p$, such that for some complex numbers $c_{ik}$
we have
$$f_i=\sum_{k=1}^pc_{ik}\CHI{X_k}$$
for all $i=1,\dots,n$. We get
\begin{eqnarray*} \begin{split}
&\sum_{i=1}^n\sum_{j=1}^n E(f_i^*f_j)(v_i,v_j)
=\sum_{i=1}^n\sum_{j=1}^n\sum_{k=1}^pE(\overline{c_{ik}}c_{jk}\CHI{X_k})(v_i,v_j)
\\&\sum_{k=1}^p\sum_{i=1}^n\sum_{j=1}^nE(\CHI{X_k})(c_{ik}v_i,c_{jk}v_j)
=\sum_{k=1}^pE(\CHI{X_k})\Big(\sum_{i=1}^nc_{ik}v_i,\sum_{j=1}^nc_{jk}v_j\Big)\geq0.
\end{split}
\end{eqnarray*}
By polarization, each linear map $f\mapsto E(f)(v,v')$ is a linear
combination of four maps of the form $f\mapsto E(f)(\un v,\un v)$
and hence continuous. Now every element $\un f_i\in\cal F$,
$i=1,\dots,n$, can be approximated in norm by functions $f_i$ of the
type considered above, and by continuity we may conclude that
$$\sum_{i=1}^n\sum_{j=1}^n E(\un f_i^*\un f_j)(v_i,v_j)\geq0.$$
Finally, (iii) trivially implies (i).
\end{proof}

\begin{theorem} If $B$ is a commutative $C^*$-algebra, then every
positive linear map $E:B\to S_\C(V;A)$ is completely positive.
\end{theorem}
\begin{proof} We may assume that $B=C_0(X)$ for some locally
compact Hausdorff space $X$. Let $\cal B$ denote the Borel
$\sigma$-algebra of $X$ and $\cal F_X$ the commutative $C^*$-algebra
of bounded $\cal B$-measurable functions on $X$. The
Riesz-Markov-Kakutani representation theorem yields a canonical
embedding of $\cal F_X$ into $B^{**}$. The restriction of $\tilde
E:B^{**}\to S_\C(V;{A^{**}})$ to $\cal F_X$ is a positive linear
map, and so it is completely positive by Lemma \ref{lemma:facp}. Its
restriction to $C_0(X)$ is just $E$ (when we consider
$S_\C(V;A)\subseteq S_\C(V;{A^{**}})$), and so $E$ is completely
positive.
\end{proof}

\subsection{The commutative case: an example}

In this subsection, we generalize the concept of a sesquilinear form
measure \cite{HPY1,HPY2} using certain modules over von Neumann
algebras. Sesquilinear form measures are used in quantum mechanics
to describe measurements where all states cannot be prepared
\cite{Pe}. 
The use of von Neumann algebras has its roots in physical
applications. In the same vein,  we assume that the modules are
countably generated to get a natural generalization of a separable
Hilbert space.

Let $A$ be a von Neumann algebra, (represented as) a weakly closed
*-subalgebra of $L(K)$, the space of bounded operators
on a Hilbert space $K$, such that the identity operator $I_K$ is in
$A$. We assume that $K$ is separable and denote its inner product by
$\<\,\cdot\,|\,\cdot\,\>_K$.

A Hilbert $C^*$-module $M$ over $A$ is countably generated if there
exists a countable set $Z\subseteq M$ such that the smallest closed
submodule which contains $Z$ is $M$. By Kasparov's stabilisation
theorem \cite[Th.\ 6.2, Cor.\ 6.3]{La95}, any countably generated
Hilbert $C^*$-module over $A$ can be seen as a (fully complemented)
$A$-submodule of $H\otimes A$ where $H$ is a separable
infinite-dimensional Hilbert space. Hence, it is not very
restrictive to consider $A$-submodules of $H\otimes A$.

Let $(\Omega,\Sigma)$ be a measurable space and  $\un V$ a module
over $A$.

\begin{definition}
Let $E:\,\Sigma\to S_A(\un V)$ be a map. Denote $E(X)=E_X$ for all
$X\in\Sigma$. The map $E$ is an {\em $A$-sesquilinear map (valued)
measure} if, for all $v,\,v'\in \un V$ the map
$$\Sigma\ni X\mapsto E_X(v,v')\in A\subseteq L(K)$$ is an operator
measure (i.e.\ $\sigma$-additive with respect to the weak operator
topology).
\end{definition}

\begin{remark}
The representation  $A\subseteq L(K)$ is here taken to be part of the
basic setting, and the Orlicz-Pettis theorem \cite{DS} shows that in
this definition the weak operator topology could be replaced by the
strong operator topology. Since any operator measure is norm bounded,
the weak operator topology could also be replaced with the $\sigma$-weak topology which has an intrinsic meaning for $A$
independently of the representation in $L(K)$.
\end{remark}

Let $H$ be a separable Hilbert space with the inner product
$\<\,\cdot\,|\,\cdot\,\>_H$ and an orthonormal basis
$\{e_n\}_{n\in\NN}$ where $\NN\subseteq\N$. Now
$V:=\lin_{\C}\{e_n\otimes a\ |\ n\in\NN,\,a\in A\}$ is an inner
product $A$-module where the module product is defined by $(h\otimes
a)\cdot a':=h\otimes (aa')$ and the inner product is $\<h\otimes
a|h'\otimes a'\>:=\<h|h'\>_Ha^*a'. $ Let $H\otimes A$ denote the
completion of $V$ to a Hilbert $C^*$-module. If $H$ is finite dimensional
then $V=H\otimes A$ is isomorphic to $A^{\dim H}$.


Let $E:\,\Sigma\to S_A(V)$ be an $A$-sesquilinear map measure. For
all $m,\,n\in\NN$  we define an operator measure
$$
\Sigma\ni X\mapsto E_{mn}(X):=E_X(e_m\otimes I_K,e_n\otimes I_K)\in
A \subseteq L(K),
$$
so that, for all $X\in\Sigma$, $m,\,n\in\NN$ and $a,\,a'\in A$,
$$
E_X(e_m\otimes a,e_n\otimes a')=a^*E_{mn}(X)a'.
$$
Thus, the structure of $E$ is determined by the operator measures
$E_{mn}$, $m,\,n\in\NN$.

Since $A$ is a subalgebra of $L(K)$ and $K$ is separable, it follows
that, for all $m,\,n\in\NN$, the operator measure $E_{mn}$ is
absolutely continuous with respect to the scalar measure
$\mu_{mn}:\,\Sigma\to\C$ \cite{HPY2}. Let $(p_{mn})_{m,n\in\NN}$ be
a double sequence of positive numbers such that $\sum_{m,n\in\NN}
p_{mn}|\mu_{mn}|(\Omega)$ converges absolutely (here $|\mu|$ means
the total variation of a complex measure $\mu$) so that
$$
X\mapsto\bm\mu(X):=\sum_{m,n\in\NN} p_{mn}|\mu_{mn}|(X)
$$
is a finite positive measure on $(\Omega,\Sigma)$ such that any
$\mu_{mn}$ is absolutely continuous with respect to $\bm\mu$
\cite{HPY2}.

Suppose then that there exists an orthonormal basis
$\{f_k\}_{k=1}^{\dim K}$ of $K$ such that as a vector subspace
$W:=\lin_\C\{f_k\,|\,k=1,2,\ldots,\dim K\}$ is invariant with
respect to $A$ (viewed as an operator algebra), that is, $aw\in W$
for all $a\in A$ and $w\in W$. Another way of saying this is that
the matrix of any $a\in A$ is column (and thus row) finite with
respect to the basis $\{f_k\}_{k=1}^{\dim K}$. This holds, for
example, when $\dim K<\infty$ or $A$ is a finitely generated algebra,
that is, there exists a finite set $G\subseteq A$, such that any
$a\in A$ is a linear combination of the elements $g_1g_2\cdots g_n$,
$n\in\N$, where $g_i\in G$ for all $i=1,\ldots,n$ \cite{EG}.

Let then $C_{mn}:\,\Omega\to S_\C(W;\C)$ be a density of $E_{mn}$
with respect to $\bm\mu$ \cite{HPY2}. We can write
$$
[C_{mn}(x)](w,w')=\sum_{k,l=1}^{\dim
K}[C_{mn}(x)](f_k,f_l)\overline{c_k}c'_l
$$
where $x\in\Omega$, $w=\sum_{k}c_k f_k$, $w'=\sum_{l}c'_l f_l\in W$
and the sums are finite. Now
$$
{ \<w|E_X(e_m\otimes a,e_n\otimes a')w'\>_K=\<a
w|E_{mn}(X)a'w'\>_K=\int_X[C_{mn}(x)](aw,a'w')d\bm\mu(x)}
$$
so that we have obtained a density for $E$, in the spirit of \cite{HPY1}.

\begin{remark}
Let $E:\,\Sigma\to S_A(V)$ be an $A$-sesquilinear map measure.
Suppose that $E(X)\ge 0$ for all $X\in\Sigma$. By integration (in
the sense of the weak operator topology) we get for all $v,\,v'\in
V$ and any $f\in\mathcal F$ an operator $E(f)(v,v') = \int_\Omega
f\,dE(\cdot)(v,v')$. For $v\in V$ the map $f\mapsto E(f)(v,v)$ is a
positive linear map, hence norm continuous, if $X\mapsto E(X)(v,v)$
is  a positive operator measure. It thus follows from Lemma 5.1 that
if $E(X)\ge 0$ for all $X\in\Sigma$ then the associated mapping
$E:\,\mathcal F\to S_A(V)$ is completely positive by Lemma
\ref{lemma:facp} and, by Theorem \ref{p1}, it has a minimal
dilation. This is a generalization of Theorem 3.6 of \cite{HPY1}.
\end{remark}



\end{document}